\documentclass[12pt,draftcls,onecolumn,letter]{IEEEtran}
\IEEEoverridecommandlockouts

\usepackage{lineno}
\usepackage{graphicx}
\usepackage{epsfig}
\usepackage{subfig}
\usepackage{times}
\usepackage{amsmath}
\usepackage{amssymb}
\usepackage{hyperref}

\usepackage{cite}
\usepackage{color}

\newtheorem{remark}{Remark}[section]
\newtheorem{lem}{Lemma}[section]
\newtheorem{theorem}{Theorem}[section]

\newtheorem{assumption}{Assumption}[section]

\newtheorem{prob}{Problem}[section]

\title{Power Generation and Distribution via Distributed Coordination Control}

\author{\small Byeong-Yeon Kim$^\dag$, Kwang-Kyo Oh$^\dag$, and Hyo-Sung Ahn$^\dag$
\thanks{\small $^\dag$School of Mechatronics, Gwangju Institute of Science and Technology, Gwangju, South Korea.
{E-mail: hyosung@gist.ac.kr}}}

\begin{document}

\maketitle


\begin{abstract}
This paper presents power coordination, power generation, and power flow control
schemes for supply-demand balance in distributed grid networks.
Consensus schemes using only local information are employed to
generate power coordination, power generation and power flow control signals. For
the supply-demand balance, it is required to determine the amount
of power needed at each distributed power node. Also due to the
different power generation capacities of each power
node, coordination of power flows among distributed power resources
is essentially required. Thus, this paper proposes a decentralized power
coordination scheme, a power generation, and a power flow control method considering
these constraints based on distributed consensus algorithms. Through
numerical simulations, the effectiveness of the proposed approaches is
illustrated.
\end{abstract}

\section{Introduction}  \label{Introduction}
In recent years, a smart grid has attracted a tremendous amount of
research interest due to its potential benefit to modern
civilization. The remarkable development of computer and
communication technology has enabled a realization of smart power
grid. A key feature of a smart power grid is the change of
the power distribution characteristics from a centralized power
system to a distributed power system. 
Though centralized power plants still cover the major portion of
power demand, the amount of power demand covered by 
distributed power resources has been increasing steadily
\cite{moslehi2010reliability}.

In distributed power systems, achieving the supply-demand balance, which is one of the fundamental requirements, is a key challenging issue due to its
decentralized characteristics. The supply-demand balance problem has been 
traditionally considered as the economic dispatch \cite{streiffert1995multi,zhang2011decentralizing}
which minimizes the total cost of operation of generation systems.
However, the traditional dispatch problem is a highly nonlinear
optimization problem and thus it has been addressed usually in a centralized manner.

Distributed control of distributed
power systems has been considered in \cite{yasuda2003basic,xin2011self,dominguez2010coordination}.
In the distributed power system, individual power resources are
interconnected with each other through communication and
transmission line.
Thus, for the coordination of power generation
and power flow, sharing information is essential. It is more
realistic and economic to use only local information.
This kind of problem has been solved in consensus
\cite{jadbabaie2003coordination,olfati2004consensus,zhu2010discrete,hui2008distributed}.
Distributed resource allocation method named ``center-free algorithm'' has been considered in \cite{servi1980electrical}. In the algorithm, since the amount of resource for each node is determined by the sum of weighted difference with its neighbors, we can easily see the equivalence of the algorithm to the consensus. Thus, we can consider that the idea of consensus had been already used in the distributed resource allocation problems. 

In this paper, power distribution among distributed power resources
is studied. The problem is classified into three subproblems.
In the first problem, which is called \textit{power coordination}, the desired net power\footnote{The ``net power'' at a node is the sum of generated power at the node and  power flows into it from its neighbor nodes. Detailed explanation on ``net power'' can be found in the second paragraph of the problem formulation section (Section 2).}
is determined. 
In the second problem and the third problem, we consider power generation and power flow
control of distributed power resources with and without power coordination respectively.
Note that the power flow control without power coordination is significantly important
when the net power and the generation capacities of each power resource are limited. Consensus schemes are
employed to generate power coordination, power generation and power flow control
signals.

Subsequently, the main contribution of this paper can be summarized
as follows. First, the coordination and control of power generation and power flow are
precisely defined and formulated. Second, two power distribution schemes with and without
power coordination are developed, which may be helpful in
coordinating actual power generation and power flow among distributed power plants.
The proposed approach deals with the desired net power which is not necessarily an average value as in the distributed averaging problem \cite{baric2011distributed}.
Third, it is shown that the distributed consensus algorithms ensure power
distribution among distributed nodes that have net power and
generation capacity constraints. The proposed approach can be applied even if
the desired net power of some power resources does not satisfy the power generation capacity
as opposed to those in \cite{dominguez2010coordination,robbins2011control}.

This paper is organized as follows. In Section \ref{sec_prob_form},
power coordination, power generation and power flow control problems are formulated.
Main results of this paper are presented in Section
\ref{sec_main_res}. Illustrative examples are provided in Section
\ref{sec_example} and the conclusion is given in Section
\ref{sec_conc}, respectively.

\section{Problem formulation} \label{sec_prob_form}
Fig. 1 shows a graph representing interconnected power systems.
\begin{figure}[t]
\centering \includegraphics [scale=0.45]{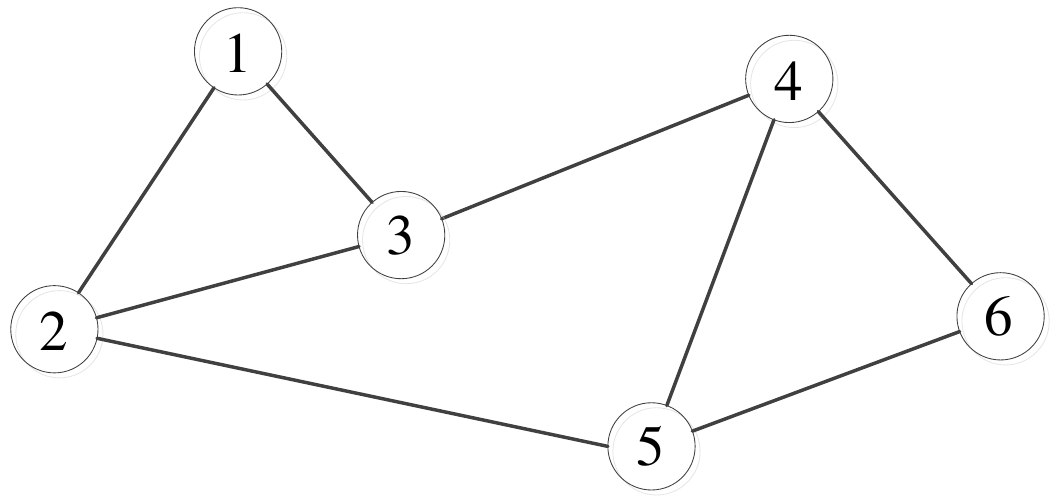}
\caption{Interconnected power systems} \label{fig1}
\end{figure}
Each node denotes a subsystem, i.e., an area that consists of generators and loads, and
it has its power demand determined by its loads. 
Thus, the desired net power for each generator in the subsystem is required to be determined.
The power can be transmitted to its neighboring subsystems since they are interconnected
with each other through transmission lines. In this paper, we would
like to control the power generation and power flow to satisfy supply-demand balance
under net power and generation capacity constraints.

Let ${p_{{i}}^d}$ be the desired net power at the $i$-th node;
${p_{{G_i}}}$ be the generated power at the $i$-th node; and
${p_{F_{ij}}}$ be the power flow from the $i$-th node to the $j$-th
node, where ${p_{F_{ij}}}= - {p_{F_{ij}}}$ and $i, j =1, \cdots, n$. 
{The net power at a node is determined after the net power
exchange of generated powers with its neighbor nodes.} Thus, the net power at the $i$-th
node can be defined as
$p_i \triangleq {p_{{G_i}}} + \sum\limits_{j \in {N_i}}
{p_{F_{ji}}}$, where $N_i$ is the set of neighbor nodes of the
$i$-th node. The total generated power and total power demand
within the power grid are represented by ${p_{G}}=\sum\limits_{i =
1}^n {{p}_{G_i}}$ and ${p_{D}}$,
respectively.

Throughout this paper, we need the following assumption to represent
a reality of power grid network systems.
\begin{assumption} \label{assump}\
\begin{enumerate}
\item Each node has the net power capacity constraint such as
\begin{equation}
\underline{{p}_{i}} \le {p_{i}}\left( k \right) \le
\overline{{p}_{i}} \label{eq1}
\end{equation}
where $k$ indicates the sampling instant of the physical power layer.
 The net power capacity is associated with the amount of power that
each node can handle. Thus, the desired net power for each node shall be
limited by the net power capacity constraint \eqref{eq1}.
\item Each node has the generation capacity constraint such as
\begin{equation}
\underline {{p_{{G_i}}}} \le {p_{{G_i}}}\left( k \right) \le
\overline {{p_{{G_i}}}} \label{eq2}
\end{equation}
The generation capacity is also assumed to be bounded by the net
power capacity constraint i.e., $\left[ {\underline {{p_{{G_i}}}}
,\overline {{p_{{G_i}}}} } \right] \subseteq \left[ {\underline
{{p_i}} ,\overline {{p_i}} } \right]$.
\item The graph which represents interconnections among nodes is undirected and connected.
This assumption comes from the fact that the power and
information flows are bidirectional. However, this assumption can be
relaxed to a directed graph.
\end{enumerate}
\end{assumption}

Note that due to the constraints \eqref{eq1} and \eqref{eq2} in the \emph{Assumption \ref{assump}}, the power coordination and power flow control problems become non-trivial.
Fig. \ref{fig2} shows a power exchange between pairs of nodes.
As shown in Fig. \ref{fig2}, the generated power and the net power at each node can be represented as follows:
\begin{align}
&{p_{{G_i}}}\left( k \right) = {p_{{G_i}}}\left( {k - 1} \right) + {\Delta p_{{G_{i}}}}\left( k \right) \label{eq8} \\
&{p_i}\left( {k} \right) = {p_{{G_i}}}\left( k \right) + \sum\limits_{j \in {N_i}} {{p_{{F_{ji}}}}\left( k \right)} ,i = 1,2, \ldots ,n \label{eq9} \\
&\sum\limits_{i = 1}^n {{p_i}\left( {k} \right)}  = \sum\limits_{i =
1}^n {{p_{{G_i}}}\left( k \right)} \label{eq10}
\end{align}
where ${\Delta p_{{G_{i}}}}\left( k \right) = {p_{G_i}}\left( k \right) -
{p_{G_i}}\left( k-1 \right)$ is its own variation
in power generation and ${p_{{F_{ji}}}}$ is the
interaction term representing the power exchange between two
nodes.
\begin{figure} [t]
\centering \includegraphics [scale=0.45]{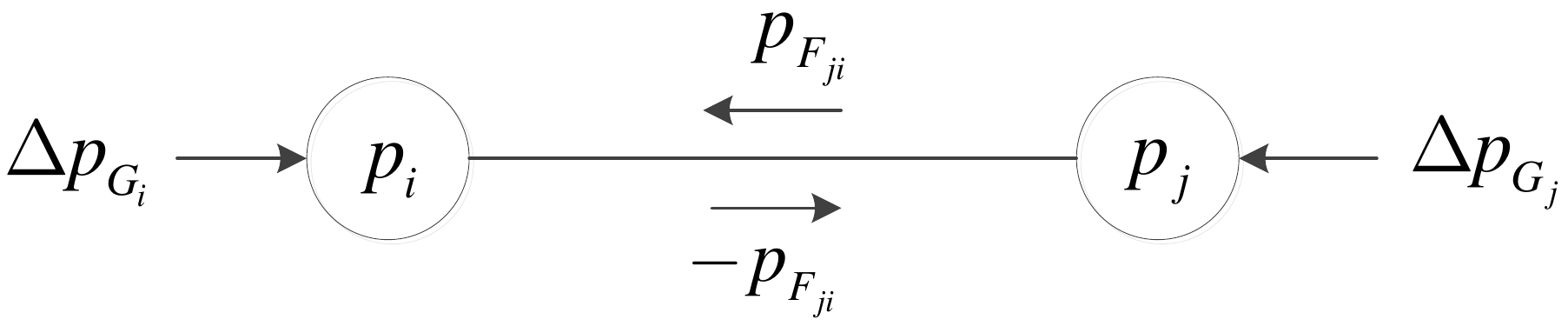} \caption{Power
exchange between pairs of nodes} \label{fig2}
\end{figure}
To make the paper clear, the following problems are formally stated.
\begin{prob}\label{prob1} {\it (Power coordination)} For a given total power demand ${p_D}\left( k
\right)$ within the network, determine the desired net power
$p_{i}^d\left( k \right)$
of individual node that satisfies the
supply-demand balance such as
\begin{equation}
{p_G}\left( k \right) = {p_D}\left( k \right) =
\sum\limits_{i = 1}^n {p_{i}^d\left( k \right)} \label{eq11}
\end{equation}
under the constraints of \emph{Assumption \ref{assump}}.
\end{prob}

\begin{prob}\label{prob2} {\it (Power generation and power flow control with power coordination)} 
For the desired net power given by the power coordination,
design power generation and power flow control strategies, i.e, ${\Delta p_{{G_{i}}}}\left( k
\right)$ and ${p_{{F_{ij}}}}\left( k \right)$, to satisfy
${p_{{i}}}\left( k \right) \to p_{i}^d\left( k \right)$ under
the constraints of \emph{Assumption \ref{assump}}.
\end{prob}

The \emph{Problem \ref{prob2}} is considered in a subsystem i.e., an area which consists of generators and loads. Thus, it is assumed that the total power demand is given as usually assumed in traditional power dispatch problems. As shown in below, only one node is required to know the total power demand. Thus, the desired net power for each node is determined by power coordination, and power generation is controlled to meet the supply-demand balance. This scenario is formulated under the name of ``With power coordination".

\begin{prob}\label{prob3} {\it (Power generation and power flow control without power coordination)} 
For the desired net power which is individually given without power coordination,
design power generation and power flow control strategies, i.e, ${\Delta p_{{G_{i}}}}\left( k
\right)$ and ${p_{{F_{ij}}}}\left( k \right)$, to satisfy
${p_{{i}}}\left( k \right) \to p_{i}^d\left( k \right)$ under
the constraints of \emph{Assumption \ref{assump}}.
\end{prob}

The \emph{Problem \ref{prob3}} is considered in interconnected subsystems under the assumption that each area has only one collective generator and one collective load. Thus, it is assumed that the desired net power is individually given to each node. Thus, power coordination is not necessary, but the power generation and power flow are controlled to meet the supply-demand balance and given-individual desired net power for each node simultaneously. This scenario is formulated under the name of ``Without power coordination''.


\begin{remark}
Power system transmission lines have a very high $X/R$ ratio. 
Thus, the real power changes are primarily dependent on phase angle differences among generators while they are relatively not affected by voltages. Further, the phase angle dynamics is generally much faster than the voltage dynamics and thus it is usually assumed that the phase angle dynamics and the voltage dynamics are decoupled \cite{saadat1999power}.
In our problem setup, since we deal with the real power, one can see that the phase angle is implicitly considered. Also, in this paper, we mainly concentrate on power distribution, transmission, and coordination issues in operation level. Thus, the lower level signal characteristics are not considered. Also, it is assumed that power flow information can be calculated from the sequence of voltage and current phasors obtained by phasor measurement units (PMUs) \cite{de2010synchronized}.
\end{remark}

\begin{remark}
The problem formulated by (\ref{eq8}), (\ref{eq9}), and (\ref{eq10}) can cover various attribute distribution problems such as gas, water and oil operation, and supply chain management, traffic control, and renewable energy allocation. This paper just focuses on power distribution problem, which has been recently more researched \cite{yasuda2003basic,xin2011self,dominguez2010coordination}.   
\end{remark}

\section{Main results} \label{sec_main_res}

\subsection{Power coordination}
In the power coordination, a desired net power is required to be realizable \cite{ahn2010command},
i.e., physical constraints such as generation
capacity and power flow constraint are required to be satisfied. If the
desired net power does not satisfy these constraints, we cannot
achieve the control goal with any control input since the desired net power is not realizable. 
From
\eqref{eq2} and \eqref{eq11}, the following condition for a
realization 
(i.e., to make the desired net power realizable)
can be obtained:
\begin{align}
\underline{{p_G}} \le {p_{D}}\left( k \right) \le \overline{{p_G}}
\label{eq12}
\end{align}
Let us consider the power coordination issue more systematically.
If the desired net power is both upper- and lower-bounded, then we can
make the following rule for the power coordination.
\begin{lem} If the desired net power is
lower-bounded and upper-bounded, i.e., ${\underline {p_{G_{i}}}} \le
p_{i}^d\left( k \right) \le {\overline {p_{G_{i}}}}$ and \eqref{eq12} is satisfied, then the
following desired net power
\begin{equation}
p_{i}^d\left( k \right) = {\underline {p_{G_{i}}}} +
\frac{{{\overline {p_{G_{i}}}} - {\underline
{p_{G_{i}}}}}}{{{\overline {p_G}} - {\underline {p_G}}}} \left(
{{p_D}\left( k \right) - {\underline {p_G}}} \right) \label{eq15}
\end{equation}
provides a sufficient condition for the supply-demand balance
\eqref{eq11}.
\end{lem}
\begin{proof}
If $p_{{i}}^d\left( k \right)$ is given by \eqref{eq15}, then,
under the assumption \eqref{eq12}, we have
\begin{align}
\underline {{p_{{G_i}}}}  - p_{{i}}^d\left( k \right)&=
\underline {{p_{{G_i}}}}  - \underline {{p_{{G_i}}}}  -
\frac{{\overline {{p_{{G_i}}}}  - \underline {{p_{{G_i}}}} }}
{{\overline {{p_G}}  - \underline {{p_G}} }} \left( {{p_D}\left( k \right) - \underline {{p_G}} } \right)
\nonumber\\
&=  - \frac{{\overline {{p_{{G_i}}}}  - \underline {{p_{{G_i}}}} }}
{{\overline {{p_G}}  - \underline {{p_G}} }} \left( {{p_D}\left( k
\right) - \underline {{p_G}} } \right) \le 0 \label{eq17}
\end{align}
which satisfies the left-side inequality ${\underline {p_{G_{i}}}}
\le p_{{i}}^d\left( k \right)$. Similarly, if $p_{{i}}^d\left( k
\right)$ is given by \eqref{eq15}, then, under the assumption
\eqref{eq12}, we obtain
\begin{align}
p_{{i}}^d\left( k \right) - \overline {{p_{{G_i}}}} 
&= \underline {{p_{{G_i}}}}  + \frac{{\overline {{p_{{G_i}}}}  -
\underline {{p_{{G_i}}}} }}
{{\overline {{p_G}}  - \underline {{p_G}} }} \left( {{p_D}\left( k \right) - \underline {{p_G}} } \right) - \overline {{p_{{G_i}}}} \nonumber\\
&= \frac{{\left( {\overline {{p_G}}  \!-\! \underline {{p_G}} }
\right)\left( {\underline {{p_{{G_i}}}}  \!-\! \overline {{p_{{G_i}}}} }
\right) \!+\! \left( {\overline {{p_{{G_i}}}}  \!-\! \underline
{{p_{{G_i}}}} } \right)\left( {{p_D}\left( k \right) \!-\! \underline
{{p_G}} } \right)}}
{{\overline {{p_G}}  - \underline {{p_G}} }} \nonumber\\
&= \frac{{\left( {\overline {{p_{{G_i}}}}  - \underline
{{p_{{G_i}}}} } \right)\left( {{p_D}\left( k \right) - \overline
{{p_G}} } \right)}} {{\overline {{p_G}}  - \underline {{p_G}} }} \le
0 \label{eq20}
\end{align}
Thus, the right-side inequality $p_{{i}}^d\left( k \right) \le
{\overline {p_{G_{i}}}}$ is also satisfied. 
Furthermore, summing
up the both sides of \eqref{eq15} over $i$ from $1$ to $n$ yields
\begin{align}
\sum\limits_{i = 1}^n {p_{i}^d\left( k \right)}  &=
\sum\limits_{i = 1}^n {\underline {{p_{{G_i}}}} }  +
\frac{{\sum\limits_{i = 1}^n {\left( {\overline {{p_{{G_i}}}}  -
\underline {{p_{{G_i}}}} } \right)} }} {{\overline {{p_G}}  -
\underline {{p_G}} }}\left( {{p_D}\left( k \right) - \underline
{{p_G}} } \right) \nonumber\\
&= {p_D}\left( k \right) \label{eq21}
\end{align}  
\end{proof}


Note that the power coordination \eqref{eq15} requires global
information such as ${\underline {p_G}}$ and ${\overline {p_G}}$.
The coordination scheme can be, however, achieved by distributed
consensus algorithm using only local information as in
\cite{cady2011robust}. For that, we need the following lemma for a further investigation.

\begin{lem} \label{lem3} \cite{horn1990matrix}
 If $A \in {\bf R}^{n \times n}$ is nonnegative and primitive, then
\begin{equation}
\mathop {\lim }\limits_{m \to \infty } {\left[ {\rho {{\left( A
\right)}^{ - 1}}A} \right]^m} \to L > 0 \label{eq22}
\end{equation}
where $L \triangleq r{l^T}$, $Ar = \rho \left( A \right)r$, $A^Tl = \rho
\left( A \right)l$, $r > 0$ and $l > 0$ in element-wise, and $r^Tl = 1$ (In fact, $r$ and $l$ are the right and left eigenvectors corresponding to the eigenvalue $\rho \left( A \right)$).
\end{lem}

Now, we state one of the main results of this subsection.
\begin{theorem} \label{thm1}
The power coordination \eqref{eq15} is achieved if the following
consensus scheme is used:
\begin{equation}
p_{i}^d\left( k \right) = \underline {{p_{{G_i}}}}  + \left(
{\frac{{\overline {{p_{{G_i}}}}  - \underline {{p_{{G_i}}}} }}
{{{y_{i,ss}}}}} \right){{x_{i,ss}}} \label{eq23}
\end{equation}
where ${x_{i,ss}}$ and ${y_{i,ss}}$ are steady state solutions of the
following equations:
\begin{align}
&{x_i}\left( {t + 1} \right) = \frac{1} {{1 + {\left| {{N_i}}
\right|} }}{x_i}\left( t \right) + \sum\limits_{j \in {N_i}}
{\frac{1}
{{1 + {\left| {{N_j}} \right|}}}{x_j}\left( t \right)} \label{eq24} \\
&{y_i}\left( {t + 1} \right) = \frac{1} {{1 + {\left| {{N_i}}
\right|}}}{y_i}\left( t \right) + \sum\limits_{j \in {N_i}}
{\frac{1}
{{1 + {\left| {{N_j}} \right|}}}{y_j}\left( t \right)} \label{eq26} \end{align}
with the following initial values
\begin{align}
&{x_i}\left( 0 \right) = \left\{ {\begin{array}{*{20}{c}}
   {{p_D}\left( k \right) - \underline {{p_{{G_i}}}} } && \text{\rm{if}} ~ {i~ {\rm{ is~ the ~leading~ node}}}  \\
   { - \underline {{p_{{G_i}}}} } && \text{\rm{otherwise}}  {}  \\
\end{array}} \right. \label{eq25} \\
&{y_i}\left( 0 \right) = \overline{p_{G_{i}}} -
\underline{p_{G_{i}}} \label{eq27}
\end{align}
respectively, where $x_i(t), y_i(t) \in {\bf R}$, and ${{\left| {{N_i}} \right|} }$ is the degree of the $i$-th
node 
and the index $t$ represents the sampling instant at the consensus algorithm. The consensus algorithm (\ref{eq23})-(\ref{eq27}) is completely decoupled from the physical power layer.
\end{theorem}
\begin{proof}
Without loss of generality, it is assumed that the first node is the
leading node which knows the total power demand $p_D\left( k \right)$.
Then, the consensus scheme \eqref{eq24} and \eqref{eq25} can be
represented by
\begin{align}
&x\left( {t + 1} \right) = Qx\left( t \right) \label{eq28}\\
&x\left( 0 \right) = {\left[ {{p_D}\left( k \right) - \underline
{{p_{{G_1}}}} , - \underline {{p_{{G_2}}}} , - \underline
{{p_{{G_3}}}} , \ldots , - \underline {{p_{{G_n}}}} } \right]^T}
\label{eq29}
\end{align}
where $x(t)=[x_1(t), \ldots, x_n(t)]^T$ and $Q = \left[ {{q_{ij}}} \right]$ is nonnegative row stochastic
matrix because ${q_{ij}} = {1 \mathord{\left/
 {\vphantom {1 {\left( {1 + \left| {{N_j}} \right|} \right)}}} \right.
 \kern-\nulldelimiterspace} {\left( {1 + \left| {{N_j}} \right|} \right)}}
$ if $j \in {N_i} \cup \left\{ i \right\}$ and $0$ otherwise. The
matrix $Q$ is irreducible because the associated graph for $Q$ is
undirected and connected. Also, $Q$ is primitive because it is
irreducible and has exactly one eigenvalue of maximum modulus
$\lambda  = \rho \left( Q \right) = 1$ (see \cite{olfati2004consensus}) according to \emph{Perron-Frobenius theorem} \cite{horn1990matrix}. 
Thus, according to
\emph{Lemma \ref{lem3}}, the vector $x$ converges to its steady
state solution ${x_{ss}} = \mathop {\lim }\limits_{m \to \infty }
{Q^m}x\left( 0 \right)$ if there exists a limit of $\mathop {\lim
}\limits_{m \to \infty } {Q^m}$. According to \emph{Lemma
\ref{lem3}}, this limit exists for the primitive matrix $Q$ and the steady
state solution is given by ${x_{ss}} = r{l^T}x\left( 0 \right)$,
where $Qr = \rho \left( Q \right)r,{Q^T}l = \rho \left( Q \right)l,r
> 0,l > 0$ in element-wise, and ${r^T}l = 1$ with $l = {\bf{1}}$ (here, ${\bf{1}}$ denotes a vector with ones as its element) and $r = {\left[
{{r_1},{r_2}, \ldots ,{r_n}} \right]^T}$ which satisfies
${r^T}{\bf{1}} = 1$. Thus, this solution is given by
\begin{align}
{x_{ss}} = r{l^T}x\left( 0 \right) = {{\bf{1}}^T}x\left( 0 \right)r = \left( {{p_D}\left( k \right) -
\underline {{p_G}} } \right)r \label{eq30}
\end{align}
In the similar manner, the steady state solution for
\begin{align}
&y\left( {t + 1} \right) = Qy\left( t \right) \label{eq31}\\
&y\left( 0 \right) = {\left[ {\overline {{p_{{G_1}}}}  - \underline
{{p_{{G_1}}}} ,\overline {{p_{{G_2}}}}  - \underline {{p_{{G_2}}}} ,
\ldots ,\overline {{p_{{G_n}}}}  - \underline {{p_{{G_n}}}} }
\right]^T} \label{eq32}
\end{align}
is given by
\begin{align}
{y_{ss}} = r{l^T}y\left( 0 \right) = {{\bf{1}}^T}y\left( 0 \right)r =
\left( {\overline {{p_G}}  - \underline {{p_G}} } \right)r
\label{eq33}
\end{align}
Thus, it follows from \eqref{eq30} and \eqref{eq33} that \eqref{eq23}
 can be represented by
\begin{align}
p_{i}^d\left( k \right) 
&= \underline {{p_{{G_i}}}}  + \frac{{\overline {{p_{{G_i}}}}  - \underline {{p_{{G_i}}}} }}{{\left( {\overline {{p_G}}  - \underline {{p_G}} } \right){r_i}}}  \left( {{p_D}\left( k \right) - \underline {{p_G}} } \right){r_i} \label{eq35}\\
&= \underline {{p_{{G_i}}}}  + \frac{{\overline {{p_{{G_i}}}}  -
\underline {{p_{{G_i}}}} }}{{\overline {{p_G}}  - \underline {{p_G}}
}}  \left( {{p_D}\left( k \right) - \underline {{p_G}} } \right)
\label{eq36}
\end{align}
which
is equivalent to \eqref{eq15}.  
\end{proof}

Since the power coordination law is fully distributed, it can be applied even if some power resources are locally added to or removed from the power grid network under the realizability assumption \eqref{eq12}.

\begin{remark}
The convergence of the algorithm \eqref{eq24}-\eqref{eq27} can be ensured 
if the sampling of the algorithm is
much faster than that of the physical layer. Thus, the power coordination scheme of \eqref{eq23}-\eqref{eq27} is feasible in practice. In more detail, let $T\left(k\right)$ be the time at the $k$-th sample instant in the physical power layer. Then,
the power coordination \eqref{eq23}-\eqref{eq27} should be
completed in $\Delta k \triangleq T\left(k+1\right) -
T\left(k\right)$. Thus, the time interval $\Delta t\triangleq
T\left(t+1\right) - T\left(t\right)$ for iterations
\eqref{eq24}-\eqref{eq27} should be chosen to $\Delta t \ll \Delta
k$ so that the steady state solutions of \eqref{eq24}-\eqref{eq27}
can be obtained in $\Delta k$.
\end{remark}

\subsection{Power generation and power flow control}
As mentioned in \emph{Problem \ref{prob2}} and \emph{Problem \ref{prob3}}, we attempt to design
power generation ${\Delta p_{{G_{i}}}}\left( k \right)$ and power flow ${p_{{F_{ij}}}}\left( k
\right)$ in order to make net power be equal to the desired net power at
each node with and without power coordination, respectively. From \eqref{eq8}-\eqref{eq9}, the net power of each node can be
described as follows:
\begin{align}
&{p_i}\left( {k} \right) = {p_{{G_i}}}\left( {k - 1} \right) + {u_i}\left( k \right) \label{eq37} \\
&{u_i}\left( k \right) = {\Delta p_{{G_{i}}}}\left( k \right) + \sum\limits_{j
\in {N_i}} {p_{{F_{ji}}}}\left(k\right) \label{eq38}
\end{align}
where ${u_i}$ is the control input at the $i$-th node.

It is remarkable that, considering the generation
capacity, the power generation control input should satisfy the following
constraint
\begin{equation}
\underline {{\Delta p_{{G_{i}}}}} \left( k \right) \le
{\Delta p_{{G_{i}}}}\left( k \right) \le \overline {{\Delta p_{{G_{i}}}}} \left(
k \right)
 \label{eq47}
\end{equation}
where $\overline{\Delta p_{{G_{i}}}}\left( k \right) =
\overline{p_{{G_i}}} - {p_{{G_i}}}\left( k - 1 \right)$ and
$\underline{\Delta p_{{G_{i}}}}\left( k \right) = \underline{p_{{G_i}}} -
{p_{{G_i}}}\left( k - 1 \right)$. Define the net power flow
at the $i$-th node as follows
\begin{equation}
{p_{{F_i}}}\left( k \right) = \sum\limits_{j \in {N_i}}
{{p_{{F_{ji}}}}\left( k \right)} \label{eq41}
\end{equation}
and define the coordination error at each node as
\begin{equation}
{p_{e,i}}\left( k \right) = {p_{{i}}}\left( {k} \right) -
p_{i}^d\left( k \right) \label{eq42}
\end{equation}

In the sequel, we provide power generation and power flow control schemes with and without power coordination taking account of two different scenarios mentioned in Section \ref{sec_prob_form}. 
\subsubsection{With power coordination}
With the power coordination \eqref{eq23}-\eqref{eq27}, the desired
net power for each node should satisfy the generation capacity of each
node. From \eqref{eq42}, the coordination error is given by
\begin{equation}
{p_{e,i}}\left( k \right) = {p_{{G_i}}}\left( {k - 1} \right) +
{u_i}\left( k \right) - p_{i}^d\left( k \right) \label{eq43}
\end{equation}
Our goal is to design ${u_i}\left( k \right)$ such that the coordination
error becomes zero.
\begin{theorem}
With the power coordination \eqref{eq23}-\eqref{eq27}, if control
input is given by
\begin{equation}
{u_i}\left( k \right)
= {\Delta p_{{G_{i}}}}\left( k \right) = p_{i}^d\left( k \right) -
{p_{{G_i}}}\left( {k - 1} \right) \label{eq44}
\end{equation}
then the constraint \eqref{eq47} and ${p_i}\left( k \right) \to p_{i}^d\left( k \right)$ can be
ensured.
\end{theorem}
\begin{proof}
With the power coordination \eqref{eq23}-\eqref{eq27},
${\underline {p_{G_{i}}}} \le p_{i}^d\left( k \right) \le
{\overline {p_{G_{i}}}}$ can be achieved. Thus, we can have
\begin{equation}
\underline {{\Delta p_{{G_{i}}}}} \left( k \right) \le p_{{i}}^d\left( k
\right) - {p_{{G_i}}}\left( {k - 1} \right) \le \overline
{{\Delta p_{{G_{i}}}}} \left( k \right) \label{eq45}
\end{equation}
In \eqref{eq38}, choosing $p_{F_{ji}}\left( k \right)=0$ and by the control law of \eqref{eq44}, we have ${u_i}\left( k \right) = {\Delta p_{{G_{i}}}}\left( k \right)$. Therefore, by \eqref{eq45}, the constraint \eqref{eq47} is satisfied. Furthermore, by inserting \eqref{eq44} into \eqref{eq37}, we can achieve ${p_i}\left( k \right) \to p_{{i}}^d\left( k
\right)$.  
\end{proof}


\begin{figure*}[t]
\centering \includegraphics [scale=0.9]{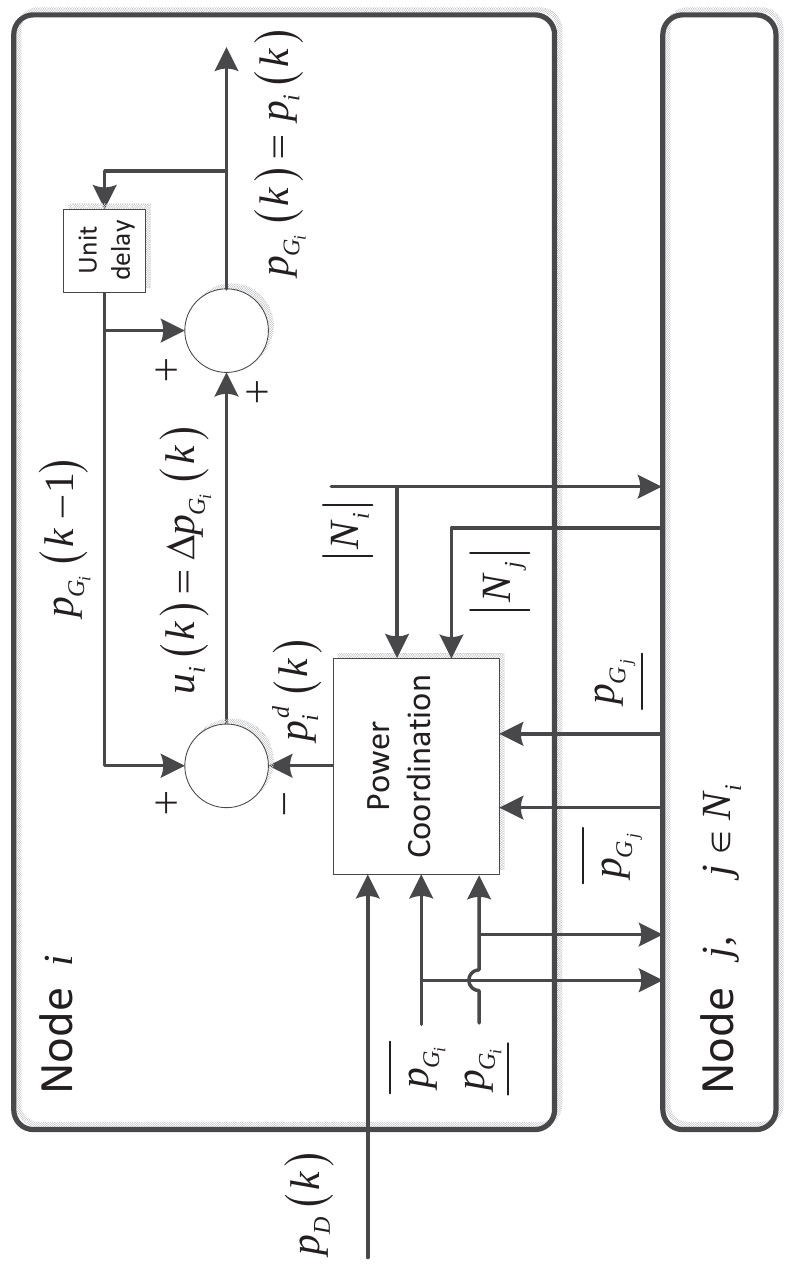} \caption{Power
generation control with power coordination} \label{fig3}
\end{figure*}

Fig. \ref{fig3} shows the overall power flow control with power coordination.
Without loss of generality, it is assumed that node $i$ is a leading
node that knows the total demand ${p_D}\left( k \right)$. Only the
information of generation capacity and the degree of node $i$ are exchanged
between neighboring nodes. With this information, the desired
net power for each node is determined by the power coordination.
After the power coordination,
power generation control input is given by \eqref{eq44}. As a
consequence of the control input, ${p_i}\left( k \right) \to
p_{i}^d\left( k \right)$ is achieved. In this case,
${p_i}\left(k\right) = {p_{G_i}}\left(k\right)$ because there is no
power flow between pairs of nodes.

\subsubsection{Without power coordination}
Let us assume that the desired net power for each node is individually given
without power coordination. But it is supposed that the desired net power satisfies the realizability
condition \eqref{eq12} and net power capacity constraints
\eqref{eq1}. In this case, the desired net power for some nodes may not
satisfy the generation capacity of their node. Thus, it is not
sufficient to provide only the power generation control for each node. Therefore,
the power flow control input can be represented as
\eqref{eq38}, with $p_{F{ji}}(k)$ enabled in this case. 
First, we want to design the power generation
control to achieve the overall supply-demand balance. 
Now, we provide a lemma to derive the main result of this paper.
\begin{lem}
If the power generation control input is designed by the following law 
\begin{align}
{\Delta p_{{G_{i}}}}\!\left( k \right) \!&=\!\underline {{\Delta p_{{G_{i}}}}}\!\left(
k \right) \!+\! \frac{{\overline {{\Delta p_{{G_{i}}}}}\!\left( k \right) \!-\!
\underline {{\Delta p_{{G_{i}}}}}\!\left( k \right)}}{{\sum\limits_{j =
1}^n {\left(\! {\overline {{\Delta p_{{G_{j}}}}}\!\left( k \right) \!-\!
\underline {{\Delta p_{{G_{j}}}}}\!\left( k \right)}\!\right)} }} \left(\!{{p_D}\!\left( k \right) \!-\! {p_G}\!\left( {k - 1} \right) \!-\! \sum\limits_{j = 1}^n \!{\underline {{\Delta p_{{G_{j}}}}}\!\left( k \right)}
} \!\right)\label{eq48}
\end{align}
then, the
control input will satisfy the constraint \eqref{eq47} and the supply-demand balance is achieved.
\end{lem}
\begin{proof}
For the left-side inequality $\underline {{\Delta p_{{G_{i}}}}} \left( k
\right) \le {{\Delta p_{{G_{i}}}}}\left( k \right)$ of \eqref{eq47}, if
${{\Delta p_{{G_{i}}}}}\left( k \right)$ is given by \eqref{eq48}, then we
can obtain the following inequality: 
\begin{align}
\underline {{\Delta p_{{G_{i}}}}} \left( k \right) - {{\Delta p_{{G_{i}}}}}\left(
k \right) &= - \frac{{\overline {{\Delta p_{{G_{i}}}}}} \left( k \right) -
\underline {{\Delta p_{{G_{i}}}}} \left( k \right)}{\sum\limits_{j =
1}^n \left( {\overline {{\Delta p_{{G_{j}}}}}} \left( k \right) -
\underline {{\Delta p_{{G_{j}}}}} \left( k \right) \right)}  \left(
{{p_D}\left( k \right) - {p_G}\left( {k - 1} \right) -
\sum\limits_{j = 1}^n {\underline {{\Delta p_{{G_{j}}}}}} \left( k \right)}
 \right) \\  
&= \frac{{\overline {{\Delta p_{{G_{i}}}}}} \left( k \right) \!-\! \underline
{{\Delta p_{{G_{i}}}}} \left( k \right)}{{\sum\limits_{j = 1}^n \!{\left(\!
{\overline {{\Delta p_{{G_{j}}}}} \left( k \right) \!-\! \underline
{{\Delta p_{{G_{j}}}}} \left( k \right)} \!\right)} }}\!\sum\limits_{j = 1}^n\!
{\left(\! {\underline {{\Delta p_{{G_{j}}}}} \left( k \right) \!-\!
{{\Delta p_{{G_{j}}}}}\left( k \right)} \!\right)}  \!\!\le\! 0
 \label{eq50}
\end{align}

For the right-side inequality
${{\Delta p_{{G_{i}}}}}\left( k \right) \le \overline {{\Delta p_{{G_{i}}}}}
\left( k \right)$, if ${{\Delta p_{{G_{i}}}}}\left( k \right)$ is given by
\eqref{eq48},  
with the substitution of $\xi  \triangleq \sum\limits_{j = 1}^n {[ {\overline {{\Delta p_{{G_{i}}}}} \left( k \right) - \underline {{\Delta p_{{G_{i}}}}} \left( k \right)} ]} [ {\underline {{\Delta p_{{G_{j}}}}} \left( k \right) - \overline {{\Delta p_{{G_{j}}}}} \left( k \right)} ] + [ {\overline {{\Delta p_{{G_{i}}}}} \left( k \right)} { - \underline {{\Delta p_{{G_{i}}}}} \left( k \right)} ] [{p_D}(k) $ $  - {p_G}(k - 1) - \sum\limits_{j = 1}^n {\underline {{\Delta p_{{G_{j}}}}} \left( k \right)} ]$,
we can obtain the following inequality:
\begin{align}
{{\Delta p_{{G_{i}}}}}\left( k \right) - \overline {{\Delta p_{{G_{i}}}}} \left( k
\right)  &= \underline {{\Delta p_{{G_{i}}}}} \left( k \right) - \overline
{{\Delta p_{{G_{i}}}}} \left( k \right)
+ \frac{{\overline {{\Delta p_{{G_{i}}}}} \left( k \right) - \underline {{\Delta p_{{G_{i}}}}} \left( k \right)}}{{\sum\limits_{j = 1}^n {\left( {\overline {{\Delta p_{{G_{j}}}}} \left( k \right) - \underline {{\Delta p_{{G_{j}}}}} \left( k \right)} \right)} }} \nonumber\\
&~~~\times \left( {{p_D}\left( k \right) - {p_G}\left( {k - 1} \right) - \sum\limits_{j = 1}^n {\underline {{\Delta p_{{G_{j}}}}} \left( k \right)} } \right) \nonumber\\
&= \frac{ \xi }
{\sum\limits_{j = 1}^n {\left( {\overline {{\Delta p_{{G_{j}}}}} \left( k \right) - \underline {{\Delta p_{{G_{j}}}}} \left( k \right)} \right)} } \nonumber\\ 
&= \frac{{\overline {{\Delta p_{{G_{i}}}}} \left( k \right) \!-\! \underline
{{\Delta p_{{G_{i}}}}} \left( k \right)}}{{\sum\limits_{j = 1}^n \!{\left(\!
{\overline {{\Delta p_{{G_{j}}}}} \left( k \right) \!-\! \underline
{{\Delta p_{{G_{j}}}}} \left( k \right)} \!\right)} }}\!\sum\limits_{j = 1}^n\!
{\left(\! {{\Delta p_{{G_{j}}}}\left( k \right) \!-\! \overline {{\Delta p_{{G_{j}}}}}
\left( k \right)}\! \right)}  \!\!\le\! 0 \label{eq54}
\end{align}
Furthermore, summing up the both sides of \eqref{eq48} over $i$ from $1$ to $n$ yields
\begin{align}
{p_D}\left( k \right) = {p_G}\left( {k - 1} \right) + \sum\limits_{i = 1}^n {\Delta {p_{G_i}}\left( k \right)}  = {p_G}\left( k \right) \label{eq53}
\end{align}
Thus, the proof is completed.  
\end{proof}

However, as in \eqref{eq15}, \eqref{eq48}
requires global information such as $\sum\limits_{i = 1}^n {\underline
{{\Delta p_{{G_{i}}}}} \left( k \right)}$, $\sum\limits_{i = 1}^n {\overline
{{\Delta p_{{G_{i}}}}} \left( k \right)}$, and ${p_D}\left( k \right) - {p_G}\left( {k - 1} \right)$. For a decentralized power generation control, we now provide the following theorem:
\begin{theorem}
The power generation control input \eqref{eq48} can be achieved if the
following consensus scheme is used.
\begin{align}
{\Delta p_{{G_{i}}}}\!\left( k \right) \!&=\! \underline {{\Delta p_{{G_{i}}}}}\! \left(
k \right) \!+\! \left(\! {\frac{{\overline {{\Delta p_{{G_{i}}}}}\! \left( k
\right) \!-\! \underline {{\Delta p_{{G_{i}}}}}\! \left( k \right)}}
{{w_{i,ss}}}}\! \right)z_{i,ss}\label{eq55}
\end{align}
where $w_{i,ss},z_{i,ss}$ are steady state solutions of the following
equations:
\begin{align}
&{z_i}\left( {t + 1} \right) = \frac{1} {{1 + {\left| {{N_i}}
\right|} }}{z_i}\left( t \right) + \sum\limits_{j \in {N_i}}
{\frac{1}
{{1 + {\left| {{N_j}} \right|} }}{z_j}\left( t \right)} \label{eq56} \\
&{z_i}\left( 0 \right) = p_{i}^d\left( k \right) - {p_G}_i\left( {k - 1} \right) - \underline {{\Delta p_{{G_{i}}}}} \left( k \right) \label{eq57} \\
&{w_i}\left( {t + 1} \right) = \frac{1} {{1 + {\left| {{N_i}}
\right|} }}{w_i}\left( t \right) + \sum\limits_{j \in {N_i}}
{\frac{1}
{{1 + {\left| {{N_j}} \right|} }}{w_j}\left( t \right)} \label{eq58} \\
&{w_i}\left( 0 \right) = \overline {{\Delta p_{{G_{i}}}}} \left( k \right)
- \underline {{\Delta p_{{G_{i}}}}} \left( k \right) \label{eq59}
\end{align}
\end{theorem}
\begin{proof}
The consensus scheme \eqref{eq56} and \eqref{eq57} can be
represented by
\begin{align}
&z\left( {t + 1} \right) = Qz\left( t \right) \label{eq60}\\
&z\left( 0 \right) = {\left[ {{z_1}\left( 0 \right),{z_2}\left( 0 \right), \ldots ,{z_n}\left( 0 \right)} \right]^T} \label{eq61}
\end{align}
As in the proof of \emph{Theorem \ref{thm1}}, 
the steady
state solution is given by ${z_{ss}} = r{l^T}z\left( 0 \right)$
where $Qr = \rho \left( Q \right)r,{Q^T}l = \rho \left( Q \right)l,r
> 0,l > 0$ in element-wise, and ${r^T}l = 1$ with $l = {\bf{1}}$ and $r = {\left[
{{r_1},{r_2}, \ldots ,{r_n}} \right]^T}$ which satisfies
${r^T}{\bf{1}} = 1$. Thus, this solution is given by
\begin{align}
\!\!\!\!{z_{ss}} \!=\! r{l^T}z\!\left( 0 \right)\!
=\!\! \left(\! {{p_D}\!\left( k \right) \!-\! {p_G}\!\left( {k - 1} \right) \!-\!
\sum\limits_{i = 1}^n {\underline {{\Delta p_{{G_{i}}}}} \!\left( k \right)}
} \!\!\right)r \label{eq64}\!\!\!
\end{align}

In the similar manner, the steady state solution for
\begin{align}
&w\left( {t + 1} \right) = Qw\left( t \right) \label{eq65}\\
&w\left( 0 \right) = {\left[ {{w_1}\left( 0 \right),{w_2}\left( 0 \right), \ldots ,{w_n}\left( 0 \right)} \right]^T} \label{eq66}
\end{align}
is given by
\begin{align}
{w_{ss}} \!=\! r{l^T}w\!\left( 0 \right)\!
=\! \sum\limits_{i = 1}^n {\left(\! {\overline {{\Delta p_{{G_{i}}}}}\! \left(
k \right) \!-\! \underline {{\Delta p_{{G_{i}}}}} \!\left( k \right)} \!\right)} r
\label{eq68}
\end{align}

Thus, \eqref{eq55} can be represented by the following equation:
\begin{align}
{\Delta p_{{G_{i}}}}\left( k \right) &= \underline {{\Delta p_{{G_{i}}}}} \left(
k \right) + \left( {\frac{{\overline {{\Delta p_{{G_{i}}}}} \left( k
\right) - \underline {{\Delta p_{{G_{i}}}}} \left( k \right)}}
{{{w_{i,ss}}}}} \right){z_{i,ss}}\nonumber\\
&= \underline {{\Delta p_{{G_{i}}}}} \left( k \right) + \frac{{\overline
{{\Delta p_{{G_{i}}}}} \left( k \right) - \underline {{\Delta p_{{G_{i}}}}}
\left( k \right)}}
{{\sum\limits_{i = 1}^n {\left( {\overline {{\Delta p_{{G_{i}}}}} \left( k \right) - \underline {{\Delta p_{{G_{i}}}}} \left( k \right)} \right)} {r_i}}} \left( {{p_D}\left( k \right) - {p_G}\left( {k - 1} \right) - \sum\limits_{i = 1}^n {\underline {{\Delta p_{{G_{i}}}}} \left( k \right)} } \right){r_i}\nonumber\\
&= \underline {{\Delta p_{{G_{i}}}}} \left( k \right) + \frac{{\overline
{{\Delta p_{{G_{i}}}}} \left( k \right) - \underline {{\Delta p_{{G_{i}}}}}
\left( k \right)}} {{\sum\limits_{i = 1}^n {\left( {\overline
{{\Delta p_{{G_{i}}}}} \left( k \right) - \underline {{\Delta p_{{G_{i}}}}}
\left( k \right)} \right)} }}\left( {{p_D}\left( k \right) -
{p_G}\left( {k - 1} \right) - \sum\limits_{i = 1}^n {\underline
{{\Delta p_{{G_{i}}}}} \left( k \right)} } \right)\label{eq71}
\end{align}
which is equivalent to \eqref{eq48}.
\end{proof}

Then, from \eqref{eq42}, the coordination error after the
power generation control input is given by
\begin{align}
{p_{e,i}}\left( k \right) 
= {p'_{e,i}}\left( k \right) + \sum\limits_{j \in {N_i}}
{{p_{{F_{ji}}}}\left( k \right)} \label{eq75}
\end{align}
where ${p'_{e,i}}\left( k \right) \buildrel \Delta \over =
{p_{{G_i}}}\left( k \right) - p_{i}^d\left( k \right)$. Now, to make ${p_{e,i}}\left( k \right) \to 0$, we need to design power flow control ${p_{{F_{ji}}}}\left( k \right)$, which is summarized in the following theorem. 


\begin{theorem}
If the power flow control is designed by
\begin{align}
{p_{{F_{ji}}}}\left( k \right) = h_{ij,ss}
\label{eq76}
\end{align}
where ${h_{ij,ss}}$ is the
steady state solution of the following equations:
\begin{align}
&{h_{ij}}\left( {t + 1} \right) = {h_{ij}}\left( t \right) + {a_{ij}}\left( {{g_j}\left( t \right) - {g_i}\left( t \right)} \right) \label{eq77}\\
&{h_{ij}}\left( 0 \right) = 0 \label{eq78}\\
&{g_i}\left( {t + 1} \right) = {g_i}\left( t \right) + \sum\limits_{j \in {N_i}} {{a_{ij}}\left( {{g_j}\left( t \right) - {g_i}\left( t \right)} \right)} \label{eq79}\\
&{g_i}\left( 0 \right) = {p'_{e,i}}\left( k \right) \label{eq80}
\end{align}
where ${a_{ij}} = \frac{1} {{1 + \max \left\{ {\left| {{N_i}}
\right|,\left| {{N_j}} \right|} \right\}}}$ is Metropolis-Hasting
weight \cite{xiao2007distributed}, then we can have ${p_{e,i}}\left( k \right) \to 0$.
\end{theorem}
\begin{proof}
From \eqref{eq79} and \eqref{eq80}, we can obtain
\begin{align}
&{g}\left( {t + 1} \right) = S{g}\left( t \right) \label{eq81}\\
&{g}\left( 0 \right) = {p'_{e}}\left( k \right) \label{eq82}
\end{align}
where ${g}\left( t \right) = {\left[ {{g_{1}}\left( t
\right),{g_{2}}\left( t \right), \ldots ,{g_{n}}\left( t \right)}
\right]^T}$ and $S = \left[ {{s_{ij}}} \right]$ is doubly stochastic,
where
\begin{equation}
{s_{ij}} = \left\{ {\begin{array}{*{20}{c}}
   {\frac{1}
{{1 + \max \left\{ {\left| {{N_i}} \right|,\left| {{N_j}} \right|} \right\}}}} & {{\text{if}}} & {j \in {N_i}}  \\
   {1 - \sum\limits_{j \in {N_i}} {\frac{1}
{{1 + \max \left\{ {\left| {{N_i}} \right|,\left| {{N_j}} \right|} \right\}}}} } & {{\text{if}}} & {i = j}  \\
   0 & {{\text{otherwise}}} & {}  \\
 \end{array} } \right.\label{eq83}
\end{equation}
and ${S^T} = S$. As in the proof of \emph{Theorem \ref{thm1}},
the steady
state solution is given by ${g_{ss}} = r{l^T}{g}\left( 0 \right)$
where $Sr = \rho \left( S \right)r,{S^T}l = \rho \left( S \right)l,r
> 0,l > 0$ in element-wise, and ${r^T}l = 1$ with $l = {\bf{1}}$ and $r$ which
satisfies ${r^T}{\bf{1}} = 1$. Furthermore, {\bf{1}} is also the right
eigenvector with the associated-eigenvalue $1$ because $S$ is doubly
stochastic. Thus, without loss of generality, let $r =
\frac{1}{n}{\mathbf{1}}$, which satisfies ${r^T}{\bf{1}} = 1$. Then,
the solution of \eqref{eq81} with \eqref{eq82} converges to the
average as follows:
\begin{equation}
{g_{ss}} = r{l^T}{g}\left( 0 \right) = \frac{1}{n}{\mathbf{1}}{{\mathbf{1}}^T}{p'_e}\left( k \right)
\label{eq84}
\end{equation}
From ${p'_{e,i}}\left( k \right) \buildrel \Delta \over =
{p_{{G_i}}}\left( k \right) - p_{i}^d\left( k \right)$,
we can obtain ${\bf{1}}^T{p'_{e}}\left( k \right) =  0
$ because of ${p_G}\left( k \right) = {p_D}\left( k \right)$. 
Thus, we have ${g_{ss}} = {\bf{0}}$. 
Hence
\begin{equation}
{g_{i,ss}} - {g_i}\left( 0 \right) =  - {p'_{e,i}}\left( k \right)
\label{eq86}
\end{equation}
Also, from \eqref{eq77}-\eqref{eq80}, it follows that
\begin{equation}
{g_{i,ss}} = {g_i}\left( 0 \right) + \sum\limits_{j \in {N_i}}
{{h_{ij,ss}}} \label{eq87}
\end{equation}
Thus, from \eqref{eq86} and \eqref{eq87}, we have
\begin{equation}
\sum\limits_{j \in {N_i}} {{h_{ij,ss}}}  =  - {p'_{e,i}}\left( k
\right) \label{eq88}
\end{equation}
If we choose the interaction control input as \eqref{eq76}, it
follows from \eqref{eq75} and \eqref{eq88} that
\begin{align}
{p_{e,i}}\left( k \right) = {p'_{e,i}}\left( k \right) + \sum\limits_{j \in {N_i}}
{{h_{ij,ss}}}  = 0 \label{eq90}
\end{align}  
\end{proof}

\begin{figure*}[!]
\centering \includegraphics [scale=0.75]{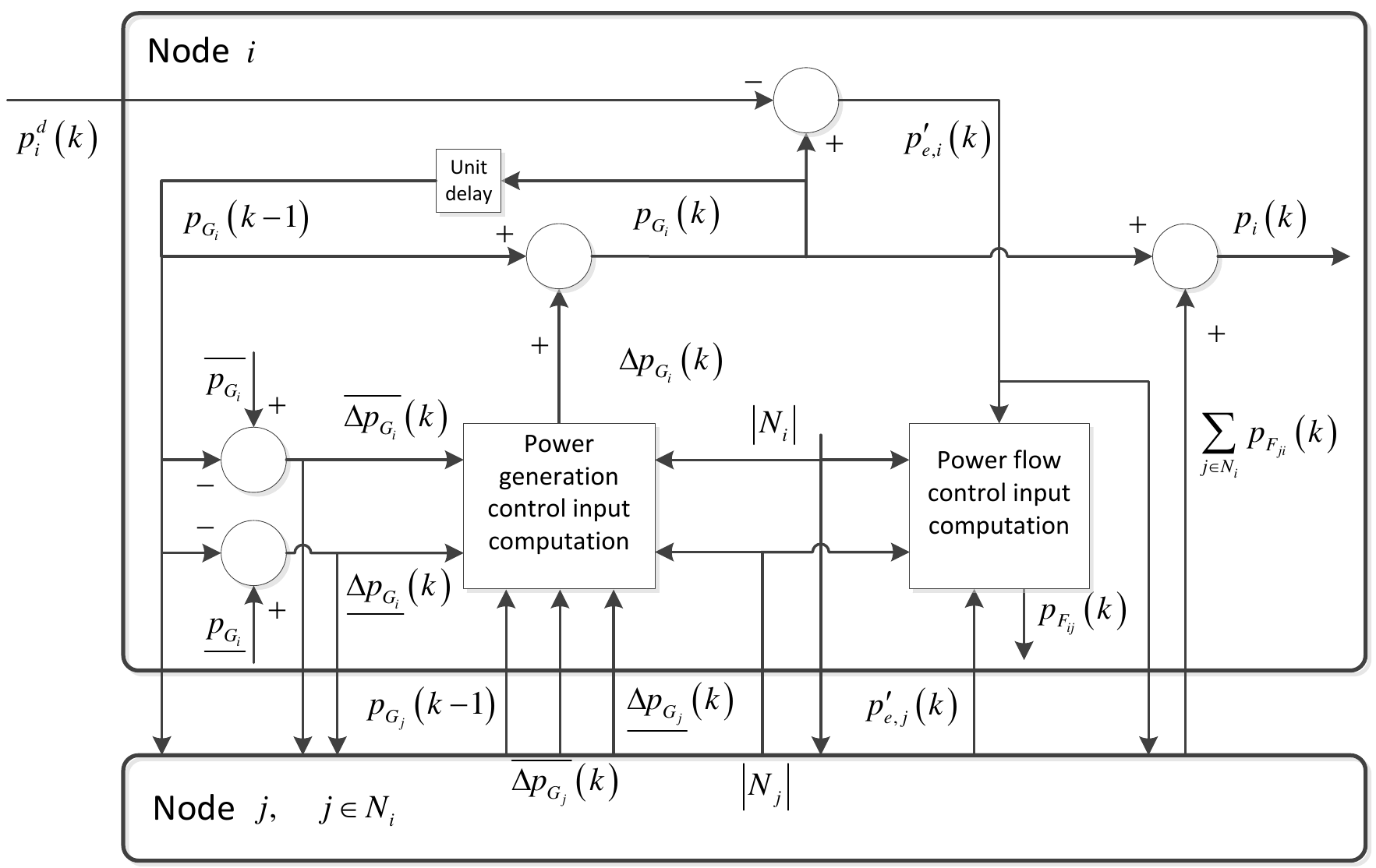} \caption{Power generation and power
flow control without power coordination} \label{fig4}
\end{figure*}

Fig. \ref{fig4} shows the power generation and power flow control scheme without
power coordination. As previously mentioned, it is assumed that the desired net power for
each node $p_{i}^d\left( k \right)$ is given under \eqref{eq12}.
The power generation control input is given by
\eqref{eq55}-\eqref{eq59}. As a consequence of the control input,
the supply-demand balance is achieved. Next, the power flow between pairs of
nodes is determined by \eqref{eq76}-\eqref{eq80}. Then, after the power flow control,
${p_i}\left( k \right) \to p_{i}^d\left( k \right)$ is
achieved. There are iterations for the power generation control input
computation and for the power flow control input computation. 
Thus, the time interval for iterations \eqref{eq56}-\eqref{eq59}
and \eqref{eq77}-\eqref{eq80} should be chosen to $\Delta t \ll
\Delta k$ so that the steady state solutions of the iterations can
be obtained in $\Delta k$.

\begin{remark}
The constraints of the \emph{Assumption 2.1} might be time-varying in renewable power resources such as a wind or a solar system.
It is possible for the proposed approach to account for time-varying bounds if the rate of change is not faster than $\Delta k$ and the realization condition \eqref{eq12} is satisfied. This can be easily verified if the time-varying bounds are substituted into the proposed approach instead of the constant bounds.
\end{remark}

\section{Illustrative examples} \label{sec_example}
In this section, two illustrative examples are provided. The
distributed power resources are interconnected as depicted
in Fig. \ref{fig1}. The generation capacity and net power capacity of
each node are listed in Table \ref{tb_1}.
\begin{table}[!htbp]
\begin{center}
\caption{Genearation capacity and net power capacity at each node}
\label{tb_1}  \normalsize 
\begin{tabular}{c|c|c}
\hline
Node$\left(i\right)$ & Generation capacity$\left({p_G}_i\right)$ & Net power capacity$\left({p_i}\right)$ \\
\hline
$1$ & $\left[ {10,50} \right]$ & $\left[ {10,80} \right]$ \\
$2$ & $\left[ {20,80} \right]$ & $\left[ {20,120} \right]$ \\
$3$ & $\left[ {20,40} \right]$ & $\left[ {20,60} \right]$ \\
$4$ & $\left[ {10,45} \right]$ & $\left[ {10,75} \right]$ \\
$5$ & $\left[ {15,60} \right]$ & $\left[ {15,90} \right]$ \\
$6$ & $\left[ {10,55} \right]$ & $\left[ {15,80} \right]$ \\
\hline
\end{tabular}
\end{center}
\end{table}

\subsection{Power distribution with power coordination}
This example shows the power distribution with power coordination. In
this case, without loss of generality, it is assumed that node $1$
is the leading node that knows the total power demand $p_D\left(k\right)$
of the distributed power system. First, the total power demand
$p_D\left(k\right)$ satisfying the realization condition \eqref{eq12} is
randomly created and it is depicted in Fig. \ref{fig5} (a). Then,
$p_{{i}}^d$ is determined by power coordination
\eqref{eq23}-\eqref{eq27} as shown in Fig. \ref{fig5} (b). Next,
the power generation control input given by \eqref{eq44} is shown in Fig.
\ref{fig5} (c) and the generated power (i.e., \eqref{eq8}) is
depicted in Fig. \ref{fig5} (d). Consequently, the coordination error given
by \eqref{eq42} is zero and the
supply-demand balance is also achieved as shown in Fig. \ref{fig5} (b) and Fig. \ref{fig5} (d).

%
%

\begin{figure}[t]
\centering \includegraphics [scale=0.5]{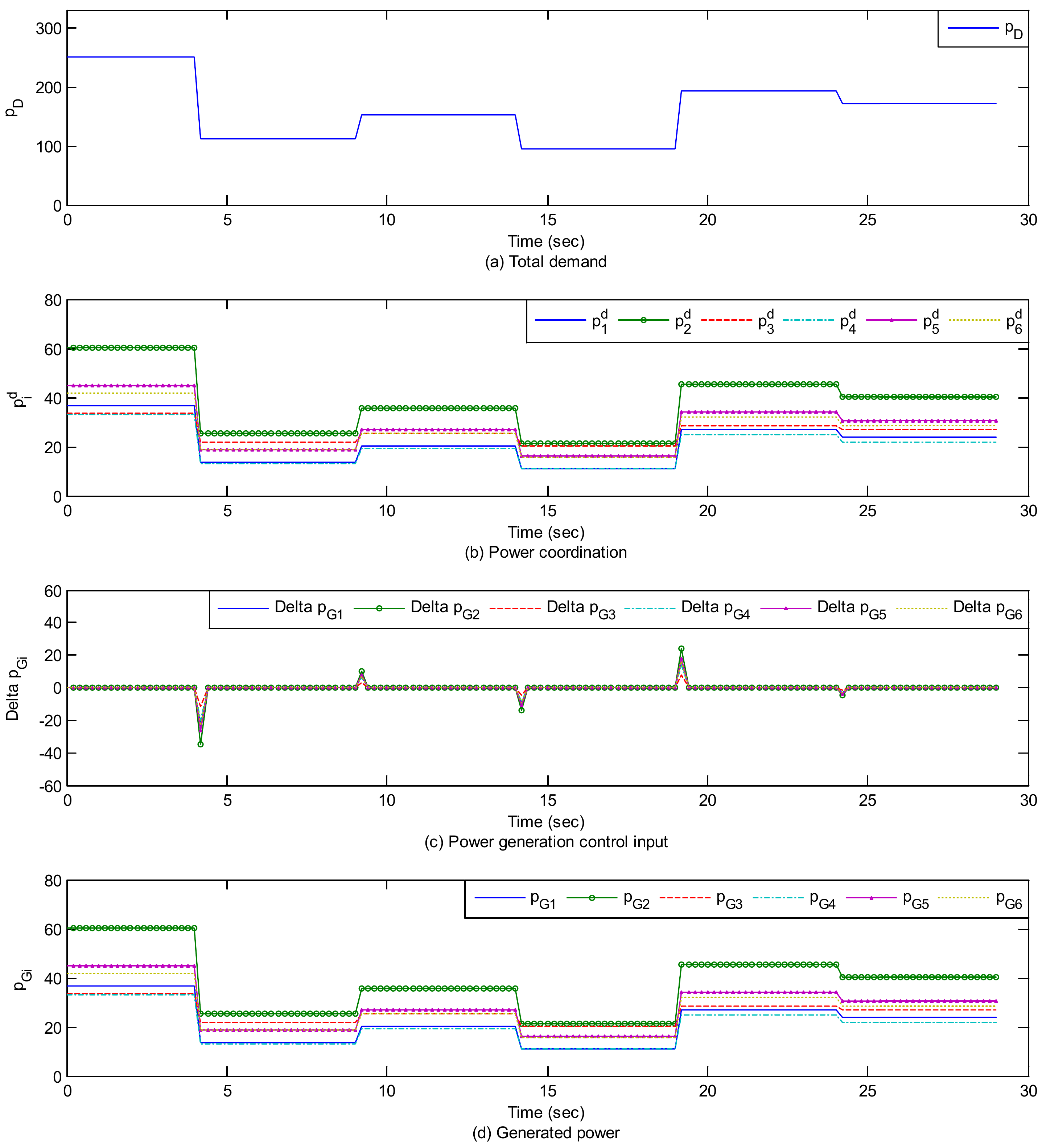}
\caption{Power distribution with power coordination} \label{fig5}
\end{figure}


\subsection{Power distribution without power coordination}

This example shows power distribution without power coordination.
In this case, the desired net power for each node is not given by
power coordination but they are randomly created under the realization
condition \eqref{eq12} as depicted in Fig. \ref{fig8} (a). First,
the power generation control input are given by \eqref{eq55}-\eqref{eq59} as shown in
Fig. \ref{fig8} (b) and the generated power (i.e., \eqref{eq8}) is 
shown in Fig. \ref{fig8} (c). After the power generation control,
a coordination of power flows between pairs of nodes is necessary to
make $p_i$ converge to $p_{{i}}^d$. The power flows between pairs
of nodes are determined by \eqref{eq76}-\eqref{eq80} and it is
depicted in Fig. \ref{fig8} (d). Next, we can obtain the net power flow at each node from
\eqref{eq41} as shown in Fig. \ref{fig8} (e). Then, the net power at
each node is given by \eqref{eq9} as shown in Fig. \ref{fig8} (f).
Then, the coordination error given by \eqref{eq42} is zero and
the supply-demand balance is also achieved as shown in Fig. \ref{fig8} (a) and Fig. \ref{fig8} (f).

\begin{figure}[t]
\centering \includegraphics [scale=0.5]{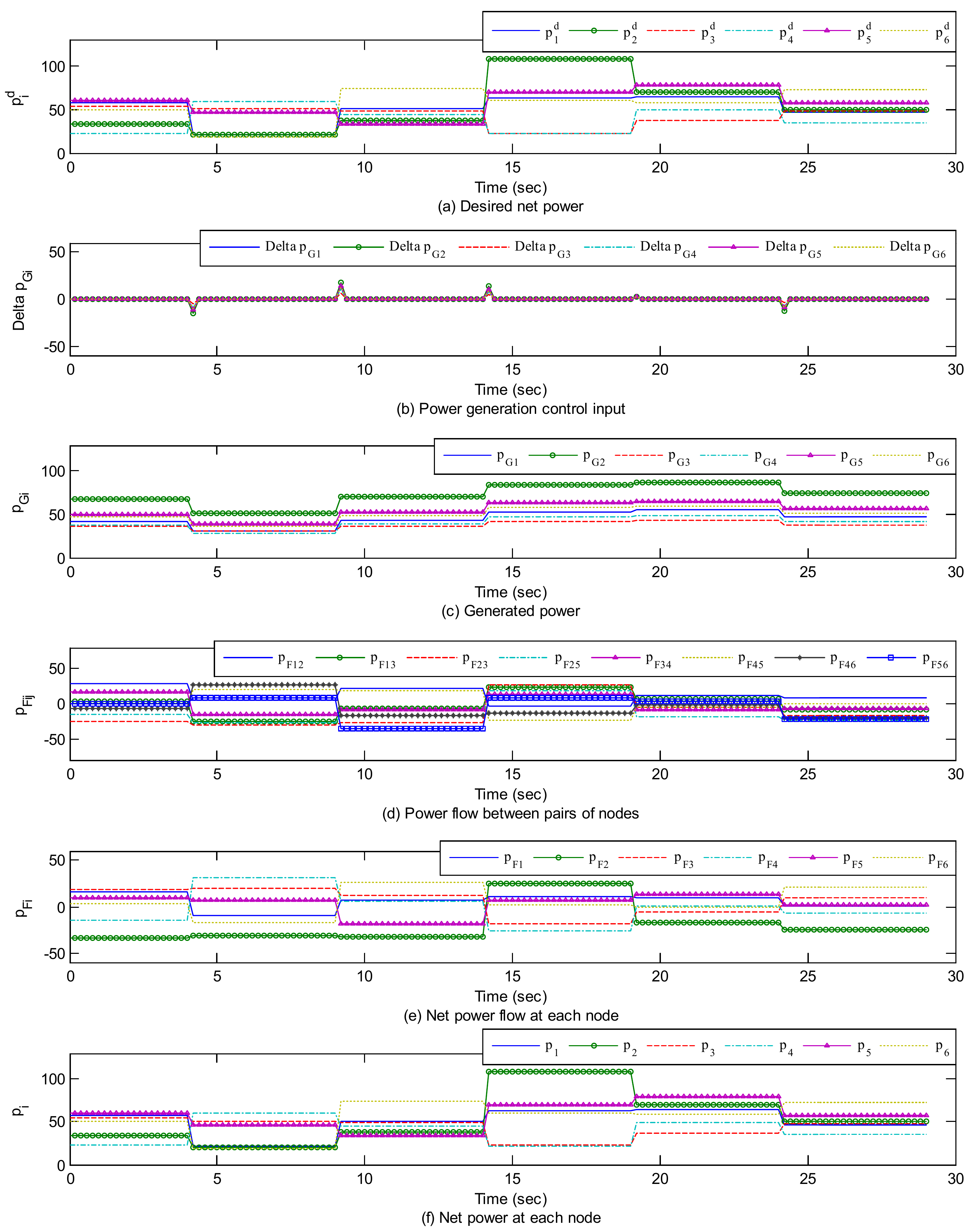}
\caption{Power distribution without power coordination} \label{fig8}
\end{figure}

%





\section{Conclusion} \label{sec_conc}
This paper has addressed power distribution problems in distributed
power grid with and without power coordination. First, a
power coordination using a consensus scheme under limited net power and power
generation capacities was considered. Second, power generation and a power flow control laws
with and without power coordination were
designed using a consensus scheme to achieve supply-demand balance.

Since this paper has provided systematic approaches for power
distribution among distributed nodes on the basis of consensus
algorithms, the results of this paper can be nicely utilized in
power dispatch or power flow scheduling. The authors believe that consensus algorithm-based
power distribution schemes of this paper have several advantages
over typical power dispatch approaches. The first key advantage is
that the power coordination and control can be realized via
decentralized control scheme without relying upon nonlinear
optimization technique. The second advantage is that the
framework proposed in this paper can handle power constraint,
generation, and flow in a unified framework. 
 
{It is noticeable that this paper has considered the power coordination,
power generation and power flow control in the higher-level models of grid networks
but does not consider lower-level models of grid networks such as current, voltage drop, and line impedance.
However, in our future research, 
it would be meaningful to add links between the higher-level models and the lower-level models.
Further, it is desirable to investigate various features such as behavior of self-interested customer, power loss in the transmission line, and structures between the physical and cyber layers such
as delays, and mismatches between them.}

\begin{remark}
Though the paper has only focused on power distribution, we believe that the proposed approach can be extended to various attribute distribution problems such as traffic control and supply chain management, because we use a fundamental model describing flow of attribute between pairs of distributed resources as well as the amount of attribute in each distributed resource. For example, in a traffic control, each freeway section named ``cell'' corresponds to each distributed power resource, distribution of vehicle density corresponds to net power of each power resource, desired traffic density corresponds to desired net power, and on-ramp traffic flow corresponds to power generation control input. Thus, the goal of traffic control which satisfies desired traffic density corresponds to that of power distribution which satisfies desired net power.
\end{remark}



\section{Acknowledgement}
It is recommended to see `{Byeong-Yeon Kim, ``Coordination and control for energy distribution using consensus algorithms in interconnected grid networks'', Ph.D. Dissertation, School of Information and Mechatronics, Gwangju Institute of Science and Technology, 2013}' for applications to various engineering problems of the algorithms developed in this paper.



\end{document}